\documentclass{amsart}

\usepackage[utf8]{inputenc} 
\usepackage[T1]{fontenc} 
\usepackage[english]{babel}

\theoremstyle{plain} 
\newtheorem{theorem}{Theorem} 
\newtheorem{lemma}[theorem]{Lemma}

\theoremstyle{remark} 
\newtheorem*{remark}{Remark} 
\newtheorem*{question}{Question}

\begin{document} 
\title{A converse to the Schwarz lemma for planar harmonic maps} 
\date{\today} 

\author{Ole Fredrik Brevig} 
\address{Department of Mathematics, University of Oslo, 0851 Oslo, Norway} 
\email{obrevig@math.uio.no}

\author{Joaquim Ortega-Cerd\`{a}} 
\address{Department de Matem\`{a}tiques i Inform\`{a}tica, Universitat de Barcelona \& Bar\-ce\-lo\-na Graduate school in mathematics, Gran Via 585, 08007 Barcelona, Spain} 
\email{jortega@ub.edu}

\author{Kristian Seip} 
\address{Department of Mathematical Sciences, Norwegian University of Science and Technology (NTNU), NO-7491 Trondheim, Norway} 
\email{kristian.seip@ntnu.no}

\begin{abstract}
	A sharp version of a recent inequality of Kovalev and Yang on the ratio of the $(H^1)^\ast$ and $H^4$ norms for certain polynomials is obtained. The inequality is applied to establish a sharp and tractable sufficient condition for the Wirtinger derivatives at the origin for harmonic self-maps of the unit disc which fix the origin. 
\end{abstract}


\thanks{Ortega-Cerdà was partially supported by the Generalitat de Catalunya (grant 2017 SGR 358) and the Spanish Ministerio de Ciencia, Innovaci\'on y Universidades (project MTM2017-83499-P). Seip was supported in part by the Research Council of Norway grant 275113}

\maketitle

\section{Introduction} 
Set $\mathbb{D}:=\{z\in\mathbb{C}\,:\,|z|<1\}$. Suppose that $f \colon \mathbb{D}\to \mathbb{D}$ is harmonic and that $f(0)=0$. Harmonic functions satisfy Laplace's equation $ \Delta f = 0$, and we write $\Delta := \partial \overline{\partial}$ for the Wirtinger derivatives
\[\partial := \frac{1}{2}\left(\frac{\partial}{\partial x}-i \frac{\partial}{\partial y}\right) \qquad \text{and} \qquad \overline{\partial} := \frac{1}{2}\left(\frac{\partial}{\partial x}+i \frac{\partial}{\partial y}\right).\]
In analogy with the classical Schwarz lemma for analytic functions, the quantities $|\partial f(0)|$ and $|\overline{\partial}f(0)|$ are of intrinsic interest. 

Set $\mathbb{T}^2:=\{z\in\mathbb{C}^2\,:\,|z_1|=|z_2|=1\}$ and let $m_2$ denote its Haar measure. Every $f$ in $L^p(\mathbb{T}^2)$ can be represented as a Fourier series $f(z) = \sum_{\alpha \in \mathbb{Z}^2} \widehat{f}(\alpha)\,z^\alpha$, where the Fourier coefficients are given by
\[\widehat{f}(\alpha) = \int_{\mathbb{T}^2} f(z)\,\overline{z^\alpha}\,dm_2(z).\]
The Hardy space $H^p(\mathbb{T}^2)$ is the subspace of $L^p(\mathbb{T}^2)$ comprised of functions $f$ such that $\widehat{f}(\alpha)=0$ unless both $\alpha_1,\alpha_2\geq0$.

A (slightly reformulated) recent result of Kovalev and Yang \cite[Thm.~1.1]{KY20} gives a description of the Wirtinger derivatives of harmonic self-maps of $\mathbb{D}$ fixing the origin in terms of the norm of certain linear functionals on $H^1(\mathbb{T}^2)$.
\begin{theorem}[Kovalev--Yang] \label{thm:KYharm} 
	Given $(\alpha,\beta)$ in $\mathbb{C}^2$, the following are equivalent. 
	\begin{enumerate}
		\item[(i)] There is a harmonic $f\colon \mathbb{D}\to\mathbb{D}$ with $f(0)=0$, $\partial f(0)=\alpha$ and $\overline{\partial} f(0)=\beta$. 
		\item[(ii)] $\|\varphi\|_{(H^1(\mathbb{T}^2))^\ast} \leq 1$ for $\varphi(z)=\alpha z_1 + \beta z_2$. 
	\end{enumerate}
\end{theorem}

Condition (ii) is in general very difficult to check, so more tractable necessary and sufficient conditions in terms of the modulus of $\alpha$ and $\beta$ are desirable. As explained in \cite[Remark~4.2]{KY20}, the following necessary condition follows from Theorem~\ref{thm:KYharm}. If $f\colon\mathbb{D}\to\mathbb{D}$ is harmonic with $f(0)=0$, $\partial f(0)=\alpha$ and $\overline{\partial} f(0)=\beta$, then 
\begin{equation}\label{eq:necharm} 
	\frac{|\alpha|+|\beta|}{2} \leq \frac{2}{\pi}. 
\end{equation}
The goal of the present note is to obtain a sharp and tractable sufficient condition similar to \eqref{eq:necharm}. 
\begin{theorem}\label{thm:suffharm} 
	If $(\alpha,\beta)\in\mathbb{C}^2$ satisfies 
	\begin{equation}\label{eq:suffharm} 
		\left(\frac{|\alpha|^4+4|\alpha\beta|^2 + |\beta|^4}{6}\right)^{\frac{1}{4}} \leq \frac{2}{\pi}, 
	\end{equation}
	then there is a harmonic $f \colon \mathbb{D}\to\mathbb{D}$ with $f(0)=0$, $\partial f(0)=\alpha$ and $\overline{\partial} f(0)=\beta$. 
\end{theorem}
\begin{remark}
	If $\alpha=\beta$, then the necessary and sufficient conditions \eqref{eq:necharm} and \eqref{eq:suffharm} coincide, which illustrates that the constant $2/\pi$ cannot be improved in either inequality. {This can also be deduced directly by considering the harmonic function
	\[f(z) = c \operatorname{Arg}\left(\frac{i-z}{i+z}\right)\]
	which maps $\mathbb{D}$ to itself if and only if $|c|\leq 2/\pi$ and which satisfies $|\alpha|=|\beta|=|c|$.} 
\end{remark}

The sufficient condition of Theorem~\ref{thm:suffharm} with $2/\pi=0.6366\ldots$ replaced by the smaller constant $5/(3+2\sqrt{6})=0.6329\ldots$ can be obtained by combining Theorem~\ref{thm:KYharm} and \cite[Thm.~1.2]{KY20}. We similarly obtain Theorem~\ref{thm:suffharm} after establishing the following sharp version of \cite[Thm.~1.2]{KY20}.
\begin{theorem}\label{thm:KY} 
	Suppose that $\varphi(z)=c_1z_1 + c_2 z_2$ for $(c_1,c_2)\neq(0,0)$. Then 
	\begin{equation}\label{eq:KYimp} 
		1 \leq \frac{\|\varphi\|_{(H^1(\mathbb{T}^2))^\ast}}{\|\varphi\|_{H^4(\mathbb{T}^2)}} \leq \frac{\pi}{2\sqrt[4]{6}}=1.0035\ldots. 
	\end{equation}
	Moreover, 
	\begin{enumerate}
		\item[(a)] the upper bound is sharp in the sense that \eqref{eq:KYimp} no longer holds if $\frac{\pi}{2\sqrt[4]{6}}$ is replaced by any smaller number. 
		\item[(b)] the lower bound is sharp in the sense that \eqref{eq:KYimp} no longer holds if $\|\varphi\|_{H^4(\mathbb{T}^2)}$ is replaced by $\|\varphi\|_{H^p(\mathbb{T}^2)}$ for any $p>4$. 
	\end{enumerate}
\end{theorem}

Comparing Theorem~\ref{thm:KY} and \cite[Thm.~1.2]{KY20}, one finds that the novelty of our result is the sharp upper bound in \eqref{eq:KYimp} and the statements (a) and (b). For the sake of completeness (and since it does not require much additional effort), we will also include a proof of the lower bound in \eqref{eq:KYimp} in our exposition.

The sharp upper bound in \eqref{eq:KYimp} is obtained after replacing an estimate due to Ramanujan on the complete elliptic integral (see \cite[pp.~6--7]{KY20}) with certain explicit expressions obtained using the Hahn--Banach theorem. As in \cite{KY20}, some polynomial estimates are required as well. Part (b) of Theorem~\ref{thm:KY} is also a consequence of the Hahn--Banach theorem along with a counter-example to a related problem from \cite{MS11}.

\subsection*{Organization} In Section~\ref{sec:prelim}, some preliminary results pertaining to $H^p(\mathbb{T}^2)$ are compiled. Section~\ref{sec:KY} is devoted to the proof of Theorem~\ref{thm:KY}. Some related work is also discussed.

\section{Preliminaries} \label{sec:prelim} 
We require certain basic properties of $H^p(\mathbb{T}^2)$. Our aim is that our note be self-contained, so we refer to broadly to the monographs \cite{Duren70,Rudin69}. Suppose that $\varphi$ is an analytic polynomial. The bounded linear functional generated by $\varphi$ on $H^p(\mathbb{T}^2)$ is 
\begin{equation}\label{eq:innerproduct} 
	L_\varphi(f) := \langle f, \varphi \rangle. 
\end{equation}
In \eqref{eq:innerproduct} and in what follows, the inner product will always denote that of $L^2(\mathbb{T}^2)$. Suppose that $1 \leq p < \infty$. We view $\varphi$ in \eqref{eq:innerproduct} as an element in $(H^p(\mathbb{T}^2))^\ast$. Hence 
\begin{equation}\label{eq:varphi} 
	\|\varphi\|_{(H^p(\mathbb{T}^2))^\ast} := \sup_{f \in H^p(\mathbb{T}^2)} \frac{|\langle f, \varphi \rangle|}{\|f\|_{H^p(\mathbb{T}^2)}} = \frac{\langle g, \varphi \rangle}{\|g\|_{H^p(\mathbb{T}^2)}} 
\end{equation}
for some $g$ in $H^p(\mathbb{T}^2)$ with $\langle g, \varphi \rangle \geq0$.

By the Hahn--Banach theorem, $L_\varphi$ extends to a bounded linear functional on $L^p(\mathbb{T}^2)$ with the same norm. Every functional on $L^p(\mathbb{T}^2)$ is of the form
\[L_\psi (f) := \langle f, \psi \rangle\]
where $\psi$ is in $L^q(\mathbb{T}^2)$ for $1/p+1/q=1$. Since the bounded linear functional $L_\psi$ extends the bounded linear functional $L_\varphi$, we must have $\|\varphi\|_{(H^p(\mathbb{T}^2))^\ast}=\|\psi\|_{L^q(\mathbb{T}^2)}$ and $P \psi = \varphi$, where $P$ is the orthogonal projection (Riesz projection) from $L^2(\mathbb{T}^2)$ to $H^2(\mathbb{T}^2)$.

In particular, we get from \eqref{eq:varphi} that
\[\|\varphi\|_{(H^p(\mathbb{T}^2))^\ast} = \frac{\langle g, \varphi \rangle}{\|g\|_{H^p(\mathbb{T}^2)}} = \frac{\langle g, \psi \rangle}{\|g\|_{L^p(\mathbb{T}^2)}}= \|\psi\|_{L^q(\mathbb{T}^2)}.\]
From the rightmost equality and H\"older's inequality, we see that $|g|^{p-2} g = C\psi$ for some constant $C>0$. Taking the Riesz projection, we conclude that 
\begin{equation}\label{eq:Pg} 
	P\big(|g|^{p-2}g\big) = C\varphi 
\end{equation}
if and only if $\varphi$ and $g$ are related as in \eqref{eq:varphi}.

Let $H^p_1(\mathbb{T}^2)$ be the two-dimensional subspace of $H^p(\mathbb{T}^2)$ consisting of functions $f(z)=a z_1 + b z_2$ for $(a,b)$ in $\mathbb{C}^2$. The orthogonal projection $P_1 \colon H^2(\mathbb{T}^2) \to H_1^2(\mathbb{T}^2)$ extends to a contraction on $H^p(\mathbb{T}^2)$ for every $1 \leq p \leq \infty$. This claim can be easily deduced from the integral representation 
\begin{equation}\label{eq:P1} 
	P_1f(z) = \int_0^{2\pi} f(e^{i\theta}z_1,e^{i\theta}z_2)\,e^{-i\theta}\,\frac{d\theta}{2\pi}. 
\end{equation}
Suppose that $\varphi(z) = c_1 z_1 + c_2 z_2$ for $(c_1,c_2)\neq(0,0)$ and that $f$ is in $H^p(\mathbb{T}^2)$. Then
\[\frac{|\langle f, \varphi \rangle|}{\|f\|_{H^p(\mathbb{T}^2)}} = \frac{|\langle P_1 f, \varphi \rangle|}{\|f\|_{H^p(\mathbb{T}^2)}} \leq \frac{|\langle P_1 f, \varphi \rangle|}{\|P_1 f\|_{H^p(\mathbb{T}^2)}}.\]
Hence, it is clear that $\|\varphi\|_{(H^p(\mathbb{T}^2))^\ast}=\|\varphi\|_{(H_1^p(\mathbb{T}^2))^\ast}$. Moreover, the optimal $g$ in \eqref{eq:varphi}, and equivalently any solution of \eqref{eq:Pg}, is of the form $g(z)=a z_1 + b z_2$.

We will next establish three results needed in the proof of Theorem~\ref{thm:KY}. The first {lemma} shows that we may swap $1$ and $4$ in the ratio appearing in Theorem~\ref{thm:KY} when considering optimal lower and upper bounds. Here and elsewhere, we use the notation $\varphi_y(z):=z_1+yz_2$, with the presumption that $0\leq y \leq 1$.
\begin{lemma}\label{lem:flipy} 
	Suppose that $\varphi(z) = c_1 z_1 + c_2 z_2$ for $(c_1,c_2)\neq(0,0)$. The estimates
	\[\inf_{0\leq y \leq 1} \frac{\|\varphi_y\|_{(H^4(\mathbb{T}^2))^\ast}}{\|\varphi_y\|_{H^1(\mathbb{T}^2)}} \leq \frac{\|\varphi\|_{(H^1(\mathbb{T}^2))^\ast}}{\|\varphi\|_{H^4(\mathbb{T}^2)}} \leq \sup_{0\leq y \leq 1} \frac{\|\varphi_y\|_{(H^4(\mathbb{T}^2))^\ast}}{\|\varphi_y\|_{H^1(\mathbb{T}^2)}}\]
	are both attained. 
\end{lemma}
\begin{proof}
	Recall that if $\varphi(z) = c_1 z_1 + c_2 z_2$, then
	\[\frac{\|\varphi\|_{(H^1(\mathbb{T}^2))^\ast}}{\|\varphi\|_{H^4(\mathbb{T}^2)}}=\frac{\|\varphi\|_{(H_1^1(\mathbb{T}^2))^\ast}}{\|\varphi\|_{H_1^4(\mathbb{T}^2)}}.\]
	The point of this reformulation is that the two-dimensional space $H^1_1(\mathbb{T}^2)$ is reflexive (while $H^1(\mathbb{T}^2)$ is not). Considering the identity operator $I\colon H^4_1(\mathbb{T}^2) \to (H^1_1(\mathbb{T}^2))^\ast$ and using duality, we find that
	\[\sup_{(c_1,c_2)\neq(0,0)} \frac{\|\varphi\|_{(H_1^1(\mathbb{T}^2))^\ast}}{\|\varphi\|_{H_1^4(\mathbb{T}^2)}} = \sup_{(c_1,c_2)\neq(0,0)} \frac{\|\varphi\|_{(H_1^4(\mathbb{T}^2))^\ast}}{\|\varphi\|_{H_1^1(\mathbb{T}^2)}}.\]
	To see that the same statement holds with $\sup$ replaced by $\inf$, consider instead $I\colon(H^1_1(\mathbb{T}^2))^\ast \to H^4_1(\mathbb{T}^2)$. Hence, we may equivalently investigate sharp upper and lower bounds for the ratio
	\[\frac{\|\varphi\|_{(H_1^4(\mathbb{T}^2))^\ast}}{\|\varphi\|_{H_1^1(\mathbb{T}^2)}} = \frac{\|\varphi\|_{(H^4(\mathbb{T}^2))^\ast}}{\|\varphi\|_{H^1(\mathbb{T}^2)}}.\]
	Set $\varphi(z)=c_1 z_1 + c_2 z_2$ for some $(c_1,c_2)\neq(0,0)$. By the rotational invariance of the Haar measure $m_2$, we may assume that $c_1,c_2\geq0$. By symmetry, we may also assume that $c_1 \geq c_2$ so $c_1>0$. Dividing $\varphi$ by a non-zero constant does not change the ratio, so with $y = c_2/c_1$, which satisfies $0 \leq y \leq 1$, we obtain
	\[\frac{\|\varphi\|_{(H^4(\mathbb{T}^2))^\ast}}{\|\varphi\|_{H^1(\mathbb{T}^2)}} = \frac{\|\varphi_y\|_{(H^4(\mathbb{T}^2))^\ast}}{\|\varphi_y\|_{H^1(\mathbb{T}^2)}}.\]
	Sharp upper and lower bounds are therefore obtained by taking the supremum and infimum, respectively, over $0 \leq y \leq 1$. 
\end{proof}

The second result readily demonstrates the virtue of the first {lemma}, since equation \eqref{eq:Pg} is easy to solve explicitly for $\varphi_y$ when $p=4$.
\begin{lemma}\label{lem:xy} 
	Given $0 \leq y \leq 1$, let $0\leq x\leq 1$ be the unique real number such that
	\[y = \sqrt{x}\frac{2+x}{1+2x},\]
	and set $g(z)=z_1+\sqrt{x}z_2$. Then $P\big(|g|^2 g\big) = (1+2x)\varphi_y$. 
\end{lemma}
\begin{proof}
	Since $|g(z)|^2 = 1+x + \sqrt{x}(z_1 \overline{z_2}+\overline{z_1}z_2)$, we find that
	\[P(|g|^2 g)(z) = (1+x)g(z)+x z_1 + \sqrt{x}z_2 = (1+2x) \left(z_1 + \sqrt{x} \frac{2+x}{1+2x}z_2\right).\]
	The proof is completed by checking that $y(x) = \sqrt{x}(2+x)/(1+2x)$ is an increasing function on $0 \leq x \leq 1$.
\end{proof}

We require the third {lemma} only for $p=1$, but we state and prove it in the general case since it requires no additional effort. Note that $\binom{1}{1/2}=4/\pi$, which explains the appearance of $\pi$ in Theorem~\ref{thm:KY}.
\begin{lemma}\label{lem:strange} 
	For $1 \leq p <\infty$, we have the identities 
	\begin{equation}\label{eq:strange} 
		\binom{p}{p/2} = \|z_1+z_2\|_{H^p(\mathbb{T}^2)}^p = \sum_{j=0}^\infty \binom{p/2}{j}^2 = \frac{4}{p} \sum_{j=0}^\infty \binom{p/2}{j}^2 j. 
	\end{equation}
\end{lemma}
\begin{proof}
	The proof relies on expressing $\|z_1+xz_2\|_{H^p(\mathbb{T}^2)}^p$, for $0 \leq x \leq 1$ in two different ways. First, we note that 
	\begin{equation}\label{eq:iexp} 
		\|z_1+xz_2\|_{H^p(\mathbb{T}^2)}^p = \|(1+x\overline{z_1}z_2)^2\|_{L^{p/2}(\mathbb{T}^2)}^{2/p} = \int_0^{2\pi} \left(1+2x\cos(\theta)+x^2\right)^{p/2}\,\frac{d\theta}{2\pi}. 
	\end{equation}
	Setting $x=1$, we obtain the first equality in \eqref{eq:strange} from a well-known integral formula for the {beta} function (see e.g. \cite[Sec.~9.3]{DurenInv}),
	\[\int_0^{2\pi} \left(2+2\cos(\theta)\right)^{p/2}\,\frac{d\theta}{2\pi} = \frac{2^{p+1}}{\pi} \int_0^{\pi/2} (\cos{\theta})^p\,d\theta = \frac{2^p}{\pi} \operatorname{B}\left(\frac{p+1}{2},\frac{1}{2}\right) = \binom{p}{p/2}.\]
	Second, we expand
	\[(1+x\overline{z_1}z_2)^{p/2} = \sum_{j=0}^\infty \binom{p/2}{j} \left(\overline{z_1}z_2\right)^j x^j.\]
	Consequently, Parseval's identity shows that 
	\begin{equation}\label{eq:bexp} 
		\|z_1+xz_2\|_{H^p(\mathbb{T}^2)}^p = \|(1+x\overline{z_1}z_2)^{p/2}\|_{L^2(\mathbb{T}^2)}^2 = \sum_{j=0}^\infty \binom{p/2}{j}^2 x^{2j}. 
	\end{equation}
	Setting $x=1$ in \eqref{eq:bexp}, we obtain the second equality in \eqref{eq:strange}. For the third equality in \eqref{eq:strange}, we differentiate the expressions \eqref{eq:iexp} and \eqref{eq:bexp} with respect to $x$ to obtain
	\[\sum_{j=0}^\infty \binom{p/2}{j}^2 2j x^{2j-1} = \frac{p}{2} \int_0^{2\pi} \left(1+2x\cos(\theta)+x^2\right)^{p/2-1}\left(2\cos(\theta)+2x\right)\,\frac{d\theta}{2\pi}.\]
	Setting $x=1$ and using \eqref{eq:iexp} yet again, we obtain the third equality in \eqref{eq:strange}. 
\end{proof}

We close the present section by explaining the connection between Theorem~\ref{thm:suffharm} and Theorem~\ref{thm:KY}. 
\begin{proof}
	[Proof of Theorem~\ref{thm:suffharm}] Suppose that $\varphi(z)=\alpha z_1 + \beta z_2$. By using the upper bound of Theorem~\ref{thm:KY}, we see that
	\[\frac{\|\varphi\|_{H^4(\mathbb{T}^2)}}{\sqrt[4]{6}} \leq \frac{2}{\pi} \qquad \implies \qquad \|\varphi\|_{(H^1(\mathbb{T}^2))^\ast} \leq 1.\]
	The proof is completed by computing $\|\varphi\|_{H^4(\mathbb{T}^2)}=\left(|\alpha|^4+4|\alpha\beta|^2+|\beta|^4\right)^{1/4}$ and appealing to Theorem~\ref{thm:KYharm}. 
\end{proof}

\section{Proof of Theorem~\ref{thm:KY}} \label{sec:KY} 
We will start from Lemma~\ref{lem:flipy}. Let $F$ be a polynomial that is strictly positive for $0 \leq y \leq 1$, and write 
\begin{equation}\label{eq:Fsneak} 
	\frac{\|\varphi_y\|_{(H^4(\mathbb{T}^2))^\ast}}{\|\varphi_y\|_{H^1(\mathbb{T}^2)}} = \frac{\|\varphi_y\|_{(H^4(\mathbb{T}^2))^\ast}}{F(y)}\frac{F(y)}{\|\varphi_y\|_{H^1(\mathbb{T}^2)}}. 
\end{equation}
Our idea is to choose $F$ in such a way that we can treat the two fractions on the right hand side of \eqref{eq:Fsneak} independently. We begin with the most technical part of the proof, which pertains to the second fraction.
\begin{lemma}\label{lem:H1inc} 
	The function 
	\begin{equation}\label{eq:H1inc} 
		\frac{1+\frac{y^2}{4}+\frac{y^4}{64}+\frac{y^6}{256}+cy^8}{\|\varphi_y\|_{H^1(\mathbb{T}^2)}} 
	\end{equation}
	is increasing on $0 \leq y \leq 1$ if and only if $c\geq \frac{5}{768}$. 
\end{lemma}
\begin{proof}
	Set $\xi = y^2$. Recall from the proof of Lemma~\ref{lem:strange} that
	\[\|\varphi_y\|_{H^1(\mathbb{T}^2)} = \sum_{j=0}^\infty \binom{1/2}{j}^2 \xi^j = 1+\frac{\xi}{4}+\frac{\xi^2}{64}+\frac{\xi^3}{256} + \sum_{j=4}^\infty \binom{1/2}{j}^2 \xi^j.\]
	We will divide by $\xi^4$ upstairs and downstairs in \eqref{eq:H1inc} and equivalently investigate
	\[R(\xi):=\frac{\Sigma_1(\xi)+c}{\Sigma_1(\xi)+\Sigma_2(\xi)}\]
	where
	\[\Sigma_1(\xi):=\sum_{j=0}^3 \binom{1/2}{j}^2 \xi^{j-4} \qquad\text{and} \qquad \Sigma_2(\xi):=\sum_{j=4}^\infty \binom{1/2}{j}^2 \xi^{j-4}.\]
	We begin by computing 
	\begin{equation}\label{eq:comp1} 
		\Sigma_1(1) = \frac{325}{256}, \qquad \Sigma_1'(1) = -\frac{1225}{256}, \qquad \Sigma_2(1) = \frac{4}{\pi}-\frac{325}{256}. 
	\end{equation}
	The two first are direct computations, while the last uses Lemma~\ref{lem:strange} to obtain that
	\[\Sigma_1(1)+\Sigma_2(1)=\sum_{j=0}^\infty \binom{1/2}{j}^2= \binom{1}{1/2} = \frac{4}{\pi}.\]
	It is clear that $\Sigma_1$ is positive and decreasing on $0 \leq y \leq 1$ and that $\Sigma_2$ is positive and increasing on $0 \leq y \leq 1$. Differentiating term by term, we find that each summand is maximized when $y=1$. Hence
	\[\big(\Sigma_1(\xi)+\Sigma_2(\xi)\big)' \leq \Sigma_1'(1)+\Sigma_2'(1) = \sum_{j=0}^\infty \binom{1/2}{j}^2 (j-4) = -\frac{15}{\pi}.\]
	In the final equality, we used Lemma~\ref{lem:strange} twice. Hence we find that $\Sigma_1'+\Sigma_2'$ is negative for $0<\xi \leq 1$ and that 
	\begin{equation}\label{eq:comp2} 
		\Sigma_2'(1) = \frac{1225}{256}-\frac{15}{\pi}. 
	\end{equation}
	We want to find a requirement on $c$ such that $R'(\xi)\geq0$ for $0 \leq \xi \leq 1$. Note that 
	\begin{equation}\label{eq:Rdiff} 
		0 \leq R' = \frac{\Sigma_1'(\Sigma_1+\Sigma_2)-(\Sigma_1+c)(\Sigma_1'+\Sigma_2')}{(\Sigma_1+\Sigma_2)^2} \quad \Longleftrightarrow \quad c \geq \frac{\Sigma_1'\Sigma_2 - \Sigma_1\Sigma_2'}{\Sigma_1'+\Sigma_2'}, 
	\end{equation}
	where we used that $\Sigma_1'(\xi)+\Sigma_2'(\xi)<0$ for $0<\xi\leq 1$. If we could prove that the right-hand side of \eqref{eq:Rdiff} is increasing on the interval $0<\xi\leq1$, then we would get the stated requirement on $c$ by \eqref{eq:comp1} and \eqref{eq:comp2}, since
	\[\frac{\Sigma_1'(1)\Sigma_2(1) - \Sigma_1(1)\Sigma_2'(1)}{\Sigma_1'(1)+\Sigma_2'(1)} = \frac{-\frac{1225}{256}\left(\frac{4}{\pi}-\frac{325}{256}\right)-\frac{325}{256}\left(\frac{1225}{256}-\frac{15}{\pi}\right)}{-\frac{15}{\pi}}=\frac{5}{768}.\]
	To prove that the right-hand side of \eqref{eq:Rdiff} is increasing on $0<\xi\leq 1$, we begin by rewriting it as 
	\begin{equation}\label{eq:rewrite} 
		\frac{\Sigma_1'\Sigma_2 - \Sigma_1\Sigma_2'}{\Sigma_1'+\Sigma_2'} = \frac{\Sigma_2 - \Sigma_2' \frac{\Sigma_1}{\Sigma_1'}}{1 + \frac{\Sigma_2'}{\Sigma_1'}}. 
	\end{equation}
	Note that 
	\begin{align*}
		0 \leq -\frac{d}{d\xi}\frac{\Sigma_1(\xi)}{\Sigma_1'(\xi)} \quad \Longleftrightarrow \quad 0 &\leq \Sigma_1''(\xi) \Sigma_1(\xi) -(\Sigma_1'(\xi))^2 . 
	\end{align*}
	The second statement can be checked directly because
	\[\Sigma_1''(\xi) \Sigma_1(\xi) -(\Sigma_1'(\xi))^2= \frac{4}{\xi^{10}} + \frac{2}{\xi^9} + \frac{11}{32 \xi^8} + \frac{5}{64 \xi^7} + \frac{17}{2048 \xi^6} + \frac{1}{4096 \xi^5} + \frac{1}{65536 \xi^4}.\]
	Hence $-\Sigma_1/\Sigma_1'$ is positive and increasing. Since both $\Sigma_2$ and $\Sigma_2'$ are positive and increasing, we find that the numerator on the right-hand side of \eqref{eq:rewrite} is increasing. Since $\Sigma_1$ is positive and decreasing and $-\Sigma_2' \Sigma_1/\Sigma_1'$ is positive and increasing, we conclude that $-\Sigma_2'/\Sigma_1'$ is positive and increasing. Consequently, the denominator on the right-hand side of \eqref{eq:rewrite} is decreasing. 
\end{proof}
We will use the polynomial
\[F(y) := 1+\frac{y^2}{4}+\frac{y^4}{64}+\frac{y^6}{256}+\frac{5}{768}y^8\]
in \eqref{eq:Fsneak}. 
\begin{remark}
	We can establish that $\|\varphi_y\|_{H^1(\mathbb{T}^2)} \leq F(y)$ for $0 \leq y \leq 1$, similarly to how the inequality
	\[\|\varphi_y\|_{H^1(\mathbb{T}^2)} \leq 1 + \frac{y^2}{4}+\frac{y^4}{64}+\frac{y^6}{128} =: \widetilde{F}(y)\]
	is proved in \cite[p.~7]{KY20}. The latter estimate is sharper for $3/5\leq y^2 \leq 1$, but the statement of Lemma~\ref{lem:H1inc} does not hold if $F$ is replaced by $\widetilde{F}$. 
\end{remark}

By Lemma~\ref{lem:H1inc}, we know that 
\begin{equation}\label{eq:Fest1} 
	\frac{F(0)}{\|\varphi_0\|_{H^2(\mathbb{T}^2)}}\leq \frac{F(y)}{\|\varphi_y\|_{H^1(\mathbb{T}^2)}} \leq \frac{F(1)}{\|\varphi_1\|_{H^2(\mathbb{T}^2)}}, 
\end{equation}
so it remains to verify that the analogous estimates hold for 
\begin{equation}\label{eq:secondfrac} 
	\frac{\|\varphi_y\|_{(H^4(\mathbb{T}^2))^\ast}}{F(y)}. 
\end{equation}
Recalling that \eqref{eq:varphi} and \eqref{eq:Pg} are equivalent, we invoke Lemma~\ref{lem:xy} to see that 
\begin{equation}\label{eq:phiyx} 
	\|\varphi_y\|_{(H^4(\mathbb{T}^2))^\ast} = \frac{\langle g, \varphi_y\rangle}{\|g\|_{H^4(\mathbb{T}^2)}} = \frac{1+\sqrt{x}y}{\left(1+4x+x^2\right)^{1/4}}=\frac{\left(1+4x+x^2\right)^{3/4}}{1+2x}, 
\end{equation}
where $y = \sqrt{x}(2+x)/(1+2x)$. By \eqref{eq:secondfrac} and \eqref{eq:phiyx}, it is equivalent to consider 
\begin{equation}\label{eq:PQ} 
	\left(\frac{(1+2x)^7\|\varphi_y\|_{(H^4(\mathbb{T}^2))^\ast}}{(1+2x)^7 F(y)}\right)^4 = \frac{(1+2x)^{28} (1+4x+x^2)^3}{(1+2x)^{32}(F(y))^4} =: \frac{P(x)}{Q(x)}. 
\end{equation}
\begin{lemma}\label{lem:PQbounds} 
	Let $P$ and $Q$ be as in \eqref{eq:PQ}. Then
	\[\frac{P(0)}{Q(0)} \leq \frac{P(x)}{Q(x)} \leq \frac{P(1)}{Q(1)}.\]
\end{lemma}
\begin{proof}
	To prove the upper bound, it is equivalent to verify that
	\[R_1(x) := \frac{P(1)Q(x)-Q(1)P(x)}{(1-x)^2}\]
	is non-negative for $0\leq x\leq 1$. We claim that $R_1$ is a polynomial of degree $46$ with positive coefficients. Hence $R_1(x)\geq0$.
	
	To prove the lower bound, we note that since $P(0)=Q(0)=1$, it is equivalent to verify that $R_2(x) := P(x)-Q(x)$ is non-negative for $0 \leq x \leq 1$. Here we claim that $R_2$ is a polynomial of degree $48$ for which the first $28$ coefficients are positive and the rest are negative. Moreover, we claim that $R_2(1)>0$. By comparing coefficients and using that $x^j \geq x^k$ for $0\leq x \leq 1$ and $0\leq j \leq k$, we deduce from this that $R_2(x)\geq0$ for $0 \leq x \leq 1$.
	
	The claims on $R_1$ and $R_2$ can be easily verified using a computer algebra system. We checked them using \texttt{Maple} and \texttt{Mathematica}. 
\end{proof}

Returning to \eqref{eq:secondfrac} and recalling that $y=y(x)$ is increasing from $y(0)=0$ to $y(1)=1$, we get from \eqref{eq:phiyx}, \eqref{eq:PQ} and Lemma~\ref{lem:PQbounds} that 
\begin{equation}\label{eq:Fest2} 
	\frac{\|\varphi_0\|_{H^4(\mathbb{T}^2)}}{F(0)} \leq \frac{\|\varphi_y\|_{(H^4(\mathbb{T}^2))^\ast}}{F(y)} \leq \frac{\|\varphi_1\|_{H^4(\mathbb{T}^2)}}{F(1)}. 
\end{equation}
\begin{proof}
	[Final part in the proof of Theorem~\ref{thm:KY}] We begin with the proof of the estimates \eqref{eq:KYimp}. By Lemma~\ref{lem:flipy}, \eqref{eq:Fsneak}, \eqref{eq:Fest1}, and \eqref{eq:Fest2}, we obtain
	\[\frac{\|\varphi_0\|_{(H^4(\mathbb{T}^2))^\ast}}{\|\varphi_0\|_{H^1(\mathbb{T}^2)}} \leq \frac{\|\varphi\|_{(H^1(\mathbb{T}^2))^\ast}}{\|\varphi\|_{H^4(\mathbb{T}^2)}} \leq \frac{\|\varphi_1\|_{(H^4(\mathbb{T}^2))^\ast}}{\|\varphi_1\|_{H^1(\mathbb{T}^2)}},\]
	where as before $\varphi_y(z)=z_1+yz_2$ and $\varphi(z) = c_1z_1+c_2z_2$ for arbitrary $(c_1,c_2)\neq(0,0)$. These estimates are evidently sharp and the lower bound is equal to $1$. To obtain a numerical value for the upper bound, we first get $\|\varphi_1\|_{H^1(\mathbb{T}^2)}=4/\pi$ from Lemma~\ref{lem:strange}. Next we use \eqref{eq:phiyx}, recalling that $y=1$ corresponds to $x=1$, to establish that $\|\varphi_1\|_{(H^4(\mathbb{T}^2))^\ast}=6^{3/4}/3$. Hence
	\[\frac{\|\varphi_1\|_{(H^1(\mathbb{T}^2))^\ast}}{\|\varphi_1\|_{H^4(\mathbb{T}^2)}} = \frac{\pi}{2 \sqrt[4]{6}},\]
	and so the proof of \eqref{eq:KYimp} and part (a) is complete. 
	
	It remains to settle (b). Suppose that the estimate 
	\begin{equation}\label{eq:p4opt} 
		\|\varphi\|_{H^p(\mathbb{T}^2)} \leq \|\varphi\|_{(H^1(\mathbb{T}^2))^\ast} 
	\end{equation}
	holds for every $\varphi(z)=c_1z_1+c_2z_2$ for some $p\geq4$. By the Hahn--Banach theorem and the fact that $(L^1(\mathbb{T}^2))^\ast=L^\infty(\mathbb{T}^2)$, we get that 
	\begin{equation}\label{eq:HBL1inf} 
		\|\varphi\|_{(H^1(\mathbb{T}^2))^\ast} = \inf_{P\psi = \varphi} \|\psi\|_{L^\infty(\mathbb{T}^2)}. 
	\end{equation}
	For $0\leq \varepsilon < 1$, consider
	\[\psi(z): = z_2 \frac{(1-\varepsilon z_1\overline{z_2})^2}{|1-\varepsilon z_1 \overline{z_2}|^2} = z_2\Bigg(-\varepsilon z_1 \overline{z_2} + \left(1-\varepsilon^2\right) \sum_{j=0}^\infty \varepsilon^j (\overline{z_1} z_2)^j\Bigg).\]
	Clearly $\|\psi\|_{L^\infty(\mathbb{T}^2)}=1$. Moreover, the Riesz projection of $\psi$ is
	\[\varphi(z) := P\psi(z)= -\varepsilon z_1 + \left(1-\varepsilon^2\right) z_2.\]
	If $\varepsilon>0$ is so small that $\varepsilon \leq 1-\varepsilon^2$, then
	\[\|\varphi\|_{H^p(\mathbb{T}^2)} = \left(1-\varepsilon^2\right)\Bigg(\sum_{j=0}^\infty \binom{p/2}{j}^2 \left(\frac{\varepsilon}{1-\varepsilon^2}\right)^{2j}\Bigg)^\frac{1}{p} = 1+\left(\frac{p}{4}-1\right)\varepsilon^2 + O(\varepsilon^4).\]
	Hence we can obtain a contradiction to \eqref{eq:p4opt} from \eqref{eq:HBL1inf} whenever $p>4$ by choosing $\varepsilon>0$ sufficiently small. We conclude that part (b) also is true. 
\end{proof}

The proof of Theorem~\ref{thm:KY} (b) is adapted from the proof of a result of Marzo and Seip \cite[Thm.~1]{MS11}, which we shall now recall. Let $P \colon L^2(\mathbb{T})\to H^2(\mathbb{T})$ denote the Riesz projection on $\mathbb{T}$. The inequality 
\begin{equation}\label{eq:MS} 
	\|Pf\|_{H^p(\mathbb{T})} \leq \|f\|_{L^\infty(\mathbb{T})} 
\end{equation}
holds for every $f$ in $L^\infty(\mathbb{T})$ if and only if $p \leq 4$. It is also demonstrated in \cite{MS11} that \eqref{eq:MS} does not hold if $\mathbb{T}$ is replaced by $\mathbb{T}^2$ and $p=4$. 

Let $P_1$ denote the operator defined by \eqref{eq:P1}. The space $P_1 L^\infty(\mathbb{T}^2)$ is comprised of essentially bounded functions on $\mathbb{T}^2$ whose Fourier coefficients are supported on the straight line $\alpha_1+\alpha_2 =1$ in $\mathbb{Z}^2$. The lower bound in Theorem~\ref{thm:KY} and its optimality can be restated as follows. The inequality 
\begin{equation}\label{eq:restate} 
	\|Pf\|_{H^p(\mathbb{T}^2)} \leq \|f\|_{L^\infty(\mathbb{T}^2)} 
\end{equation}
holds for every $f$ in $P_1 L^\infty(\mathbb{T}^2)$ if and only if $p\leq 4$. Hence one can think of \eqref{eq:restate} as a partial analogue of \eqref{eq:MS} on $\mathbb{T}^2$.

Fix $1 \leq q \leq 2$. What is the largest $2 \leq p \leq 4$ such that
\[\|\varphi\|_{H^p(\mathbb{T}^2)} \leq \|\varphi\|_{(H^q(\mathbb{T}^2))^\ast}\]
holds for every $\varphi(z)=c_1z_1+c_2z_2$? 

By adapting the counter-example from \cite[Thm.~9]{BOSZ18} similarly to how we adapted the counter-example from \cite[Thm.~1]{MS11} in the proof of Theorem~\ref{thm:KY} (b), the necessary condition $p \geq 4/q$ can be established. By the lower bound in Theorem~\ref{thm:KY}, we know that this is sharp for $q=1$. It is also trivially sharp for $q=2$. Similarly, the answer to the following question is affirmative in the endpoint cases $q=1,2$.
\begin{question}
	Fix $1 < q < 2$. Is it true that
	\[1 \leq \frac{\|\varphi\|_{(H^q(\mathbb{T}^2))^\ast}}{\|\varphi\|_{H^{4/q}(\mathbb{T}^2)}} \leq 2 \binom{q}{q/2}^{-1/q} \binom{4/q}{2/q}^{-q/4}\]
	for every $\varphi(z)=c_1z_1 + c_2 z_2$ with $(c_1,c_2)\neq(0,0)$? 
\end{question}

The upper bound in the question is obtained by setting $c_1=c_2=1$. To compute the ratio in this case, we first use \cite[Lem.~5]{Brevig19} to see that
\[\|z_1+z_2\|_{(H^q(\mathbb{T}^2))^\ast} = 2 \|z_1+z_2\|_{H^q(\mathbb{T}^2)}^{-1}\]
and then Lemma~\ref{lem:strange} twice.

\bibliographystyle{amsplain} 
\bibliography{kyineq}

\end{document}